\documentclass[final]{siamltex}
\usepackage{amsfonts}
\usepackage{amsmath}
\usepackage{enumerate}
\usepackage{graphicx}
\usepackage{multicol}
\usepackage{epsfig}

\newtheorem{example}[theorem]{Example}
\newtheorem{property}[theorem]{Property}

\title{A new generalized field of values}
\author{Ricardo Reis da Silva\thanks{Universiteit van Amsterdam, Korteweg-de Vries Instituut for Mathematics, P.O. Box 94248, 1090 GE Amsterdam ({\tt r.j.reisdasilva@uva.nl}).}}

\begin{document}

\maketitle

\begin{abstract}
Given a right eigenvector $x$ and a left eigenvector $y$ associated with the same eigenvalue of a matrix $A$, there is a Hermitian positive definite matrix $H$ for which $y=Hx$. The matrix $H$ defines an inner product and consequently also a field of values. The new generalized field of values is always the convex hull of the eigenvalues of $A$. Moreover, it is equal to the standard field of values when $A$ is normal and is a particular case of the field of values associated with non-standard inner products proposed by Givens. As a consequence, in the same way as with Hermitian matrices, the eigenvalues of non-Hermitian matrices with real spectrum can be characterized in terms of extrema of a corresponding generalized Rayleigh Quotient.
\end{abstract}

\begin{keywords}
field of values, numerical range, Rayleigh quotient
\end{keywords}

\begin{AMS}
15A60, 15A18, 15A42
\end{AMS}

\pagestyle{myheadings}
\thispagestyle{plain}
\markboth{Ricardo Reis da Silva}{A NEW GENERALIZED FIELD OF VALUES}

\section{Introduction}\label{sec:intro}
We propose a new generalized field of values which brings all the properties of the standard field of values to non-Hermitian matrices. For a matrix $A\in\mathbb{C}^{n\times n}$ the (classical) field of values or numerical range of $A$ is the set of complex numbers
\begin{equation}
	F(A)\equiv \{x^{*}Ax:x\in\mathbb{C}^{n},x^{*}x=1\}.
\end{equation}
Alternatively, the field of values of a matrix $A\in\mathbb{C}^{n\times n}$ can be described as the region in the complex plane defined by the range of the Rayleigh Quotient
\begin{equation}
	\rho(x,A)\equiv \frac{x^{*}Ax}{x^{*}x},\hspace{10mm}\forall x\neq 0.
\end{equation}
For Hermitian matrices, field of values and Rayleigh Quotient exhibit very agreeable properties: $F(A)$ is a subset of the real line and the extrema of $F(A)$ coincide with the extrema of the spectrum, $\sigma(A)$, of $A$. Equivalently, the vector $v$ maximizing $\rho(x,A)$ is an eigenvector associated with the largest eigenvalue, $\rho(v,A)$, of the matrix $A$. Therefore, every eigenpair of $A$ is the solution of a maximization-minimization problem in some constraint subspace. Such results are Rayleigh-Ritz and Courant-Fischer's theorems \cite[\S 4.2]{horn_matrix_1985}.

Unfortunately, as the next example shows, for non-Hermitian matrices such pleasant properties no longer hold. Even if all the eigenvalues of $A$ would be real.\\

\begin{example}
Figure \ref{fig:fovalsrandA} and \ref{fig:fovalsA} depict the classical field of values of a Hermitian matrix, $A$, and a non-Hermitian matrix, $B$, with real eigenvalues respectively. Although the eigenvalues of the matrix $B$ lay on the real line such as those of $A$, the extrema of the field of values and of the spectrum no longer coincide. \\

\begin{figure}[ht]
	\begin{minipage}{0.5\linewidth}
		\centering
		\includegraphics[scale=0.4]{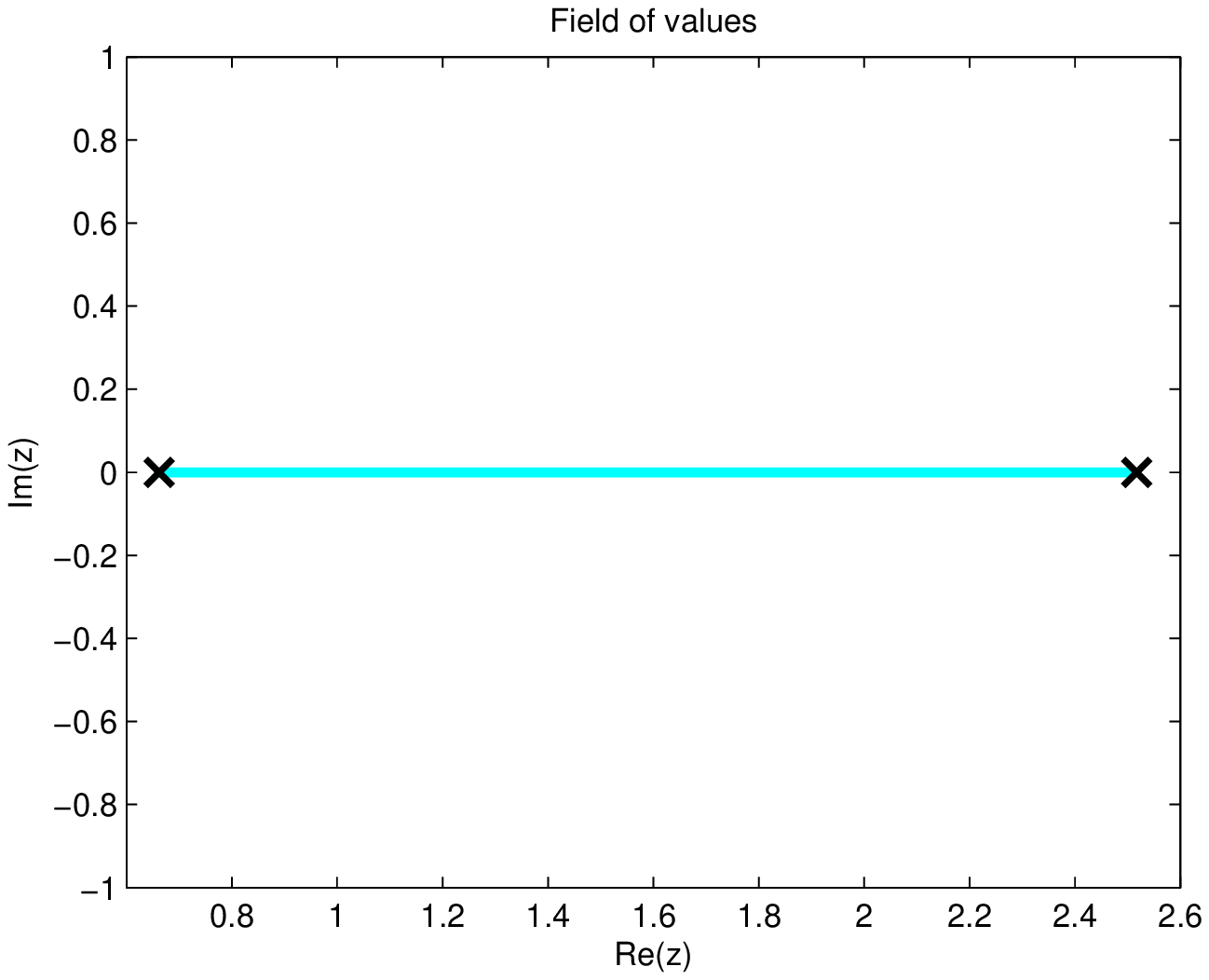}
		\caption{Field of Values of $A$, $F(A)$ }
		\label{fig:fovalsrandA}
	\end{minipage}
	\begin{minipage}{0.5\linewidth}
		\centering
		\includegraphics[scale=0.4]{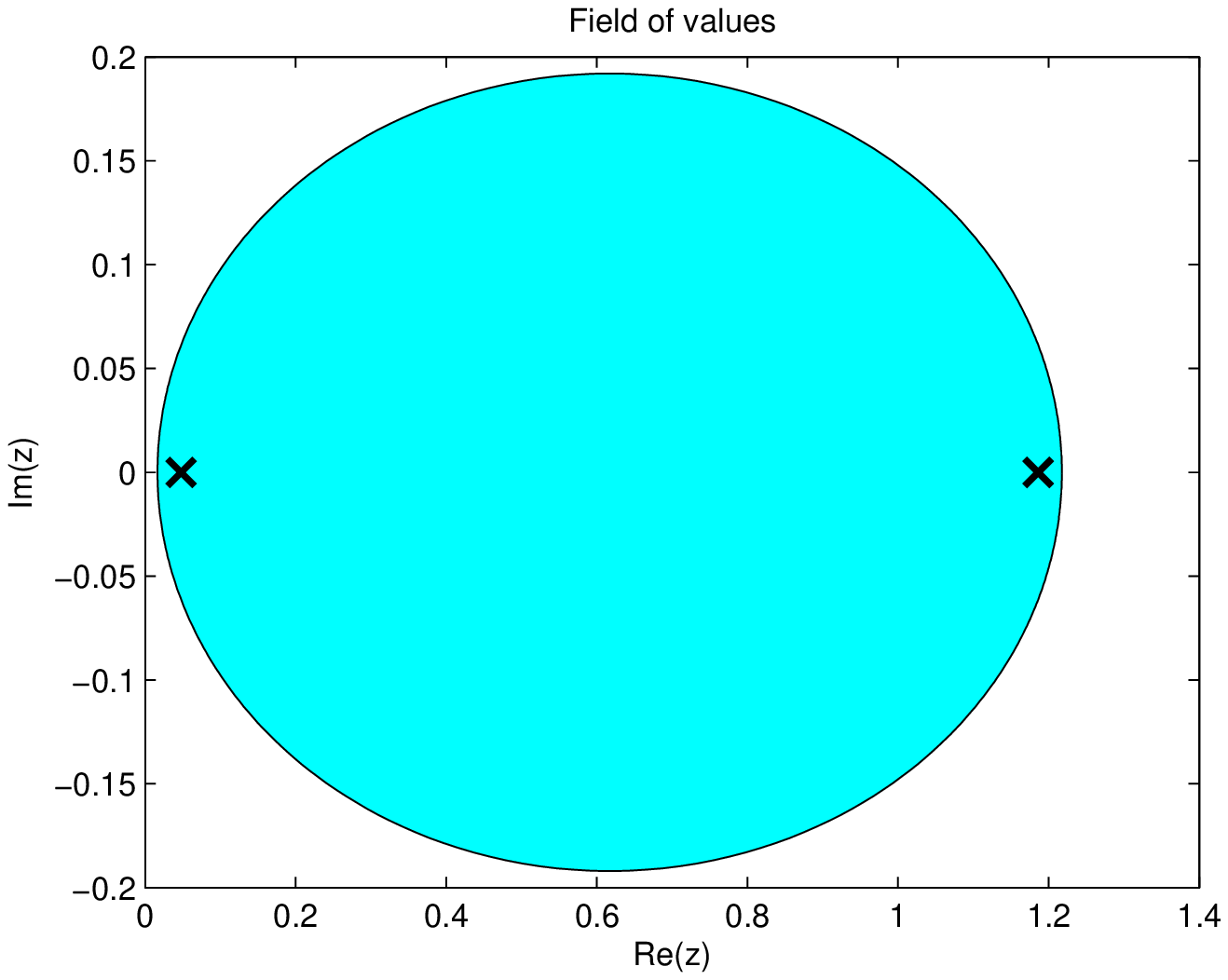}
		\caption{Field of Values of $B$, $F(B)$}
		\label{fig:fovalsA}
	\end{minipage}
\end{figure}
\end{example}

Within the last sixty years several generalizations for the field of values were proposed (see \cite[\S 1.8]{horn_topics_1991} for a more complete list). Each of the generalizations attempted to replicate one or more of the properties of the classical field of values. In 1952, Givens \cite{givens_fields_1952} proposed the field of values of $A\in\mathbb{C}^{n\times n}$ associated with a \emph{generalized inner product}. For any $A\in\mathbb{C}^{n\times n}$ Givens field of values is the set
\begin{equation}
	F_{H}(A)\equiv \{x^{*}HAx:x\in\mathbb{C}^{n\times n}, x^{*}Hx=1, H\mbox{ is Hermitian positive definite}\}.
\end{equation}
Givens intent was to extend $F(A)$'s invariance under unitary similarity transformations to arbitrary similarity transformations. If, for arbitrary nonsingular $V$, $B=VAV^{-1}$ is a similarity transformation of $A$, then $F_{H}(A)=F(B)$ for $H=V^{*}V$. The converse is also true. Givens proves two important statements:\\
\begin{theorem}\label{thm:Givens1}
The intersection of the regions $F_{H}(A)$ for all positive definite matrices $H$ is the minimum convex polygon, $Co(\sigma(A))$ containing all the roots of $A$.
\end{theorem}
\\
\begin{theorem}\label{thm:Givens2}
$F_{H}(A)=Co(\sigma(A))$ for some positive definite hermitian matrix $H$ if and only if the elementary divisors corresponding to roots lying on the boundary of $Co(\sigma(A))$ are simple.
\end{theorem}\\

When $A$ is normal $F_{H}(A)=F(A)=Co(\sigma(A))$, the convex hull of the spectrum of $A$. It is unclear, however, if for arbitrary $A$, Givens knew which $H$ (if any) satisfies $F_{H}(A)=Co(\sigma(A))$. 

A few years later, Bauer \cite{bauer_field_1962} proposed a new generalization. Bauer extended the original formulation of the field of values to norms which need not be associated with inner products
\begin{equation}
	F_{\|\cdot\|}(A)\equiv \{y^{*}Ax:x,y\in\mathbb{C}^{n}\mbox{ and }\|y\|^{D}=\|x\|=y^{*}x=1\}
\end{equation}
where $\|y\|^{D}$ stands for the dual (vector) norm (see \cite{nirschl_bauer_1964}). Bauer's generalized field of values makes use of different (though related) vectors at the right and left sides of $A$. Those vectors are dual pairs with respect to the vector norm $\|\cdot\|$. Moreover, $F_{\|\cdot\|}(A)$ depends only on the norm and $A$ and not on an inner product. However, and although Bauer's generalized field of values always contains $\sigma(A)$, it is not always convex \cite{nirschl_bauer_1964, zenger_convexity_1968}. Moreover, according to the authors in \cite{nirschl_bauer_1964} the field of values from Givens is a special case of Bauer's field of values.

A more recent generalization is the \emph{$q$-field of values} \cite{marcus_constrained_1977} 
\begin{equation}
	F_{q}(A)\equiv  \{y^{*}Ax:x,y\in\mathbb{C}^{n}, y^{*}y=x^{*}x=1,\mbox{ and }y^{*}x=q\}.
\end{equation}
Also for $F_{q}(A)$ different vectors $x$ and $y$ are used for the inner product $\langle Ax,y\rangle$. They must, however, satisfy an extra constraint $y^{*}x=q$. The convexity property holds for $q\in[0,1]$ and $A\in\mathbb{C}^{n\times n}$ with $n\geq 2$ \cite{horn_topics_1991}. Finally, if $q=1$, the $q$-field of values reduces to the standard field of values.

\subsection{Notation and Outline}

Although most of the notation we use is considered standard, we believe to be useful to clarify some assumptions. We represent eigenvalues
by the Greek letters $\lambda$ and choose to label them in non-decreasing order of magnitude:
$\min{\lambda}=\lambda_{1}\leq\lambda_{2}\leq...\leq\lambda_{n-1}\leq\lambda_{n}=\max{\lambda}$. The spectrum of a matrix $A$, the set of the eigenvalues of $A$, is denoted by $\sigma(A)$. Letters at the end of the alphabet represent vectors and these will always be column vectors. Lower case Latin letters $i$ and $j$ denote indices. The conjugate transpose of a matrix $X$ is denoted by $X^{*}$, also for $X$ real. With
$I_{n}$ we mean the $n\times n$ identity matrix and by $e_{j}$ its  $j$th column. To minimize clutter we use $A^{-*}$, if necessary, to mean $(A^{*})^{-1}=(A^{-1})^{*}$ while $\langle x ,y\rangle$ will denote the standard inner product between vectors $x,y\in\mathbb{C}^{n}$.

In \S \ref{sec:GTFV} we define the new generalized field of values giving emphasis to the relations between left and right eigenvectors of non-Hermitian matrices. In \S \ref{sec:Properties} we prove important properties of the generalized two-sided field of values and show how, for particular types of matrices, it naturally reduces to the classical field of values and to some of the early generalized approaches. Section \ref{sec:TSRQ} describes the two-sided Rayleigh Quotient illustrating some of its properties and relating it with the newly defined generalized two-sided field of values. Here we show how the properties of the classic Rayleigh Quotient carryover to the generalized field of values once the correct inner product is chosen. Finally, in that same section, we show how using the new definitions the extremal properties of the standard Rayleigh Quotient known for Hermitian matrices can be extended to the non-Hermitian case for matrices with real eigenvalues.

\section{A new generalized field of values}	\label{sec:GTFV}

If $A$ is a Hermitian matrix and $\lambda\in\mathbb{R}$ an eigenvalue of $A$, the set of vectors $v$ satisfying $Av=\lambda v$ is equal to the set of vectors $w$ satisfying $w^{*}A=\lambda w^{*}$. These are the right and left eigenvectors associated with the eigenvalue $\lambda$. If $A$ is non-Hermitian, however, that is not necessarily the case. None of the generalizations of the field of values discussed in section \S \ref{sec:intro} focuses on capturing the dynamics of left and right eigenvectors associated with the same eigenvalue of a non-Hermitian matrix. Assume $A$ is nondefective. Then there exists a nonsingular matrix $V$ whose columns are right eigenvectors of $A$ and that satisfies
\begin{equation}\nonumber
	V^{-1}A=\Lambda V^{-1}\hspace{10mm}\mbox{and}\hspace{10mm}AV=V\Lambda.
\end{equation}
The rows of $V^{-1}$ are the left eigenvectors of $A$ while $\Lambda$ is the diagonal matrix of the eigenvalues of $A$. A crucial observation is the following
\begin{lemma}\label{lem:rleigv}
	Let $v_{r},v_{l}\in\mathbb{C}^{n\times n}$ be the right and left eigenvectors corresponding to the same eigenvalue $\lambda_{i}$ of a nondefective matrix $A\in\mathbb{C}^{n\times n}$. Let $A=V\Lambda V^{-1}$ be  a diagonalization of $A$ and denote by $\Lambda$ the diagonal matrix of the eigenvalues. Take $V$ to be a matrix whose columns are eigenvectors of $A$ scaled appropriately. Then $v_{r}$ and $v_{l}$ satisfy
		\begin{equation}\nonumber	
			 V^{*}v_{l}=V^{-1}v_{r}=e_{i}\hspace{5mm}\mbox{or equivalently}\hspace{5mm}v_{r}=VV^{*}v_{l}=Ve_{i}.
		\end{equation}
\end{lemma}

\begin{proof}
	The nondefectiveness of $A$ guarantees the existence of a nonsingular matrix $V$ for which $A=V\Lambda V^{-1}$ where $\Lambda=diag(\lambda_{1},\ldots,\lambda_{n})$. As a consequence, the matrix $\Lambda$ is complex symmetric. Therefore, for $v_{r}$ and $v_{l}$ as above, 
$(v_{l}^{*}V)^{*}=V^{-1}v_{r}=e_{i}$ or equivalently, $v_{l}=(VV^{*})^{-1}v_{r}=V^{-*}e_{i}$ or yet $v_{r}=VV^{*}v_{l}=Ve_{i}$.
\end{proof}\\

The Hermitian case ($A^{*}=A$), emerges as a particular occurrence of the previous lemma. The matrix $V$ is then unitary, i.e. $V^{*}=V^{-1}$ and as a consequence the statement shortens to $v_{r}=v_{l}=Ve_{i}$. 

Similar to the case of Bauer's and of the $q$-field of values, so does our generalized two-sided field of values uses different vectors on the right and left of $A$. Those vectors $y$ and $x$ now form a dual pair with respect to an eigenvalue of $A$. The definition that follows introduces the generalized two-sided field of values.

\begin{definition}[Generalized two-sided field of values]
For a nondefective matrix $A=V\Lambda V^{-1}$, the generalized two-sided field of values is the set of complex numbers
\begin{equation}\label{def:Gtsfv}
	G(A)\equiv \{y^{*}Ax:x,y\in\mathbb{C}^{n},y=(VV^{*})^{-1}x\mbox{ and }y^{*}x=1\}.
\end{equation}
\end{definition}

\section{Properties of the generalized two-sided field of values}\label{sec:Properties}

We will assume that $A$ can be diagonalized as $A=V\Lambda V^{-1}$ and that $G(A)$ is defined as in (\ref{def:Gtsfv}). The first two properties to come are standard properties for fields of values (see also \cite{johnson_functional_1976}). Further on we introduce some particular characteristics of $G(A)$.

\begin{property}
For any complex number $\alpha$,
	\begin{equation}\nonumber
		G(A+\alpha I)=G(A)+\alpha.
	\end{equation}
\end{property}
{\em Proof}.
That the basis of eigenvectors of $A$ and $A+\alpha I$ is the same follows from $V^{-1}(A+\alpha I)V=\Lambda+\alpha I=D$. Therefore,
	\begin{eqnarray}
		G(A+\alpha I)&=&\{y^{*}(A+\alpha I)x:y=(VV^{*})^{-1}x, y^{*}x=1\}\nonumber \\
					 &=&\{y^{*}Ax+\alpha y^{*}x:y=(VV^{*})^{-1}x, y^{*}x=1\} \nonumber \\
					 &=&\{y^{*}Ax:y=(VV^{*})^{-1}x, y^{*}x=1\}+\alpha=G(A)+\alpha.\nonumber\qquad\endproof
	\end{eqnarray}

\begin{property}
	For any nonzero complex number $\alpha$,
	\begin{equation}\nonumber
		G(\alpha A)=\alpha G(A).
	\end{equation}
\end{property}
{\em Proof}.
	\begin{eqnarray}
		G(\alpha A)&=&\{y^{*}(\alpha A)x:y=(VV^{*})^{-1}x, y^{*}x=1\}\nonumber \\
				   &=&\{\alpha y^{*}Ax:y=(VV^{*})^{-1}x, y^{*}x=1\} \nonumber \\
				   &=&\alpha\{y^{*}Ax:y=(VV^{*})^{-1}x, y^{*}x=1\}=\alpha G(A).\nonumber \qquad\endproof
	\end{eqnarray}

\begin{property} 
For all nondefective matrices $A\in\mathbb{C}^{n\times n}$,
\begin{equation}\nonumber
	G(A)=F_{H}(A) 
\end{equation}
with $H=(VV^{*})^{-1}$.
\end{property}
{\em Proof}.
	Define $H=(VV^{*})^{-1}$. The matrix $H^{-1}=VV^{*}$ is positive definite and therefore so is $H$. Consequently,
	\begin{equation}\nonumber
		y^{*}Ax=x^{*}(VV^{*})^{-1}Ax=x^{*}HAx\hspace{10mm}\mbox{and}\hspace{10mm}1=y^{*}x=x^{*}Hx.\qquad\endproof
	\end{equation}

As such, $G(A)$ is a particular case of $F_{H}(A)$. However, as the next properties show it possesses very agreeable characteristics.
\begin{property}\label{prop:equivalence}
	For all nondefective $A\in\mathbb{C}^{n\times n}$ of the form $A=V\Lambda V^{-1}$,
	\begin{equation}\nonumber
		G(A)=F(\Lambda).
	\end{equation}
\end{property}

{\em Proof}.
	Let $z=V^{-1}x$. Then,
	\begin{equation}\nonumber
		y^{*}Ax=xV^{-*}V^{-1}Ax=z^{*}V^{-1}AVz=z^{*}\Lambda z \hspace{10mm}\mbox{and}\hspace{10mm}1=y^{*}x=z^{*}z.\qquad \endproof
	\end{equation}

\begin{property}\label{prop:realpart}
	Let $H(A)=\frac{1}{2}(A+A^{*})$,
	\begin{equation}\nonumber
		G(H(A))=\Re F(A).
	\end{equation}
\end{property}
\begin{proof}
	Because $H(A)$ is normal it follows from Property \ref{prop:normality} that
	\begin{equation}\nonumber
		G(H(A))=F(H(A))=\Re F(A),
	\end{equation}
	where the last equality follows from the properties of the classical field of values.
\end{proof}
Property \ref{prop:realpart} contrasts with the equivalent one for the classical field of values. Unlike the latter, the projection of $G(H(A))$ onto the real axis for non-Hermitian matrices is not orthogonal but oblique and so $G(H(A))\neq \Re G(A)$.

It follows from Property \ref{prop:equivalence} and from the properties of the classical field of values that
\begin{property}\label{prop:Comp}
	$G(A)$ is a compact and convex set.
\end{property}

Not only is the set compact and convex as it also possesses a very important property related to the spectrum of $A$.
\begin{property}\label{prop:ConvexHull}
For all nondefective $A\in\mathbb{C}^{n\times n}$
	\begin{equation}\nonumber
		G(A)=Co(\sigma(A)).
	\end{equation}
where $Co(\sigma(A))$ denotes the convex hull of the spectrum of $A$.
\end{property}
\begin{proof}
	If $A$ is nondefective, there exists a nonsingular matrix $V$ for which $A=V^{-1}\Lambda V$, where $\Lambda =diag(\lambda_{1},\ldots,\lambda_{n})$ is the diagonal matrix of eigenvalues of $A$. By definition $G(A)=F(\Lambda )$ where the field of values of $\Lambda $ is the set of all convex combinations of the diagonal elements of $\Lambda $:
	\begin{equation}\nonumber
		z^{*}\Lambda z=\sum_{i=1}^{n}|z_{i}|^{2}\lambda_{i}\hspace{10mm}\mbox{where}\hspace{10mm}\sum_{i}^{n}|z_{i}|^{2}=1.
	\end{equation}
	The proof follows from the definition of convex hull of a set $S$ as the set of all convex combinations of finitely many points of $S$.
\end{proof}
\begin{property}\label{prop:normality}
	For all normal $A\in\mathbb{C}^{n\times n}$,
	\begin{equation}\nonumber
		G(A)=F(A).
	\end{equation}
\end{property}
\begin{proof}
As a consequence of the normality of $A$ the matrix $V$ is unitary, i.e. $V^{-1}=V^{*}$. Therefore,
\begin{equation}\nonumber
	G(A)=F(V^{-1}AV)=F(V^{*}AV)=F(A)
\end{equation}
where the last equality results from the unitary invariance property of the classical field of values (\cite[Property 1.2.8]{horn_topics_1991}).
\end{proof}

In summary, $G(A)$ is, for any nondefective matrix, always convex and always the convex hull of the eigenvalues of $A$. This gives the following corollary to Theorem \ref{thm:Givens1}:
\begin{corollary}
	When $V$ is the matrix of eigenvectors of $A$ and $\bar{H}=(VV^{*})^{-1}$, the region $F_{\bar{H}}(A)$ is the intersection of the regions $F_{H}(A)$ over all positive definite $H$. 
\end{corollary}\\

The matrix $\bar{H}=(VV^{*})^{-1}$, however, is not the only matrix $H$ for which $F_{H}=Co(\sigma(A))$ as the next example shows. 
\begin{example}
Let $A$ be a nondefective matrix such that both $A=V\Lambda V^{-1}$ and $V$ can be partitioned as follows 
\begin{equation}\nonumber
	A=\left[\begin{array}{cc}
		\Lambda_{1} & \\	
			  & A_{2}
	\end{array}\right]\hspace{10mm}\mbox{and}\hspace{10mm}V=\left[\begin{array}{cc}
		I 	& \\	
			& V_{2}
	\end{array}\right]
\end{equation}
where $\Lambda_{1}$ is diagonal and $F(A_{2})\subset F(\Lambda_{1})$. Then for both $H_{1}=I$ and $H_{2}=(VV^{*})^{-1}$ ,$F_{H_{1}}=F_{H_{2}}=Co(\sigma(A))$.
\end{example}

{\em Note:} So far we have assumed that the matrix $A$ must be nondefective. However, as Theorem \ref{thm:Givens2} from Givens shows, this requirement can be made less tight. In fact, let $A$ be any (square) matrix whose elementary divisors corresponding to eigenvalues lying on the boundary of $Co(\sigma(A))$ are simple. Let $J$ be the bidiagonal matrix of the Jordan normal form of $A$ and partition it as follows
	\begin{equation}\nonumber
		J=\left[\begin{array}{cc}
			\Lambda  & \\
			  & T	
		\end{array}\right].
	\end{equation}
 Here, $\Lambda $ is the (diagonal) matrix of eigenvalues of $A$ located on the boundary of $Co(\sigma(A))$ and $T$ the bidiagonal matrix containing the remaining roots. Partition $W$ in a similar fashion as $W=[W_{1}\hspace{3mm} W_{2}]$ where $W_{1}$ contains the set of linear independent eigenvectors associated with the eigenvalues in $\Lambda $ and $W_{2}$ the set of generalized eigenvectors associated with the eigenvalues on the diagonal of $T$. Then, 
 \begin{equation}\nonumber
	G(A)=\{y^{*}Ax:x,y\in\mathbb{C}^{n},y=(WW^{*})^{-1}x,y^{*}x=1\}.
 \end{equation}

To end this section we prove a result on the location of $G(A)$ in the complex plane when $A$ is positive definite. First, however, we need to extend the definition of positive definite matrices. It is widely accepted that a positive definite matrix is a \emph{Hermitian} matrix, $A\in\mathbb{C}^{n\times n}$, for which
\begin{equation}\nonumber
	x^{*}Ax>0,\hspace{10mm}\mbox{for all }0\neq x\in\mathbb{C}^{n}.
\end{equation}
The previous definition neglects non-Hermitian matrices with (real) positive eigenvalues. Furthermore, there is no agreement on the literature on what the proper extension to the non-Hermitian case should be. We wish to contribute to the discussion by proposing a definition consistent with the earlier generalization of the field of values. In this way,
\begin{definition}[Positive definite matrix]\label{def:PDM}
	A nondefective matrix $A\in\mathbb{C}^{n\times n}$ diagonalizable as $A=V\Lambda V^{-1}$ is said to be positive semidefinite if 
	\begin{equation}\label{eq:PSDM}
		y^{*}Ax\geq 0,\hspace{10mm}\mbox{for all } 0 \neq x\in\mathbb{C}^{n}\hspace{5mm}\mbox{and}\hspace{5mm}(VV^{*})y=x,
	\end{equation}
	and positive definite if, in addition,
	\begin{equation}\label{eq:PDM}
		y^{*}Ax> 0,\hspace{10mm}\mbox{for all } 0 \neq x\in\mathbb{C}^{n}\hspace{5mm}\mbox{and}\hspace{5mm}(VV^{*})y=x.
	\end{equation}
\end{definition}
This definition is consistent with the generalization of the inner product. This is due to the fact that Equations (\ref{eq:PSDM}) and (\ref{eq:PDM}) are equivalent to
\begin{equation}
	\langle Ax,x \rangle_{H}\geq 0\hspace{10mm}\mbox{and}\hspace{10mm} \langle Ax,x \rangle_{H}> 0
\end{equation}
respectively, in the inner product defined by the matrix $H=(VV^{*})^{-1}$, the $H$-inner product. To satisfy the alternative characterization that $A$'s eigenvalues are real and positive, we must have $H=(VV^{*})^{-1}$. For normal matrices $H=I$, thus reverting to the standard definition. We now enunciate the property
\begin{property}
	If $A$ is positive definite then $G(A)\subset \{z:Re\hspace{2mm} z>0\}$.
\end{property}
\begin{proof}
	If $A$ is positive definite in the sense of the definition above then $A$ is either Hermitian positive definite or has real positive eigenvalues. That the property is true for the Hermitian case follows from Property \ref{prop:normality}. For non-Hermitian matrices the result is given by Property \ref{prop:ConvexHull}.
\end{proof}

\begin{example}
Figures \ref{fig:fovalsA2} and \ref{fig:fovalsrandA2} show the standard, $F(A)$, and the Generalized two-sided field of values of a non-Hermitian matrix $A$ with real eigenvalues. Figures \ref{fig:fovalsA3} and \ref{fig:fovalsrandA3} represent a similar situation for a random matrix $B\in\mathbb{R}^{n\times n}$ with some complex eigenvalues.
\begin{figure}
	\begin{minipage}{0.5\linewidth}
		\centering
		\includegraphics[scale=0.4]{V2FovalsA.eps}
		\caption{Standard FoV, $F(A)$}
		\label{fig:fovalsA2}
	\end{minipage}
	\begin{minipage}{0.5\linewidth}
		\centering
		\includegraphics[scale=0.4]{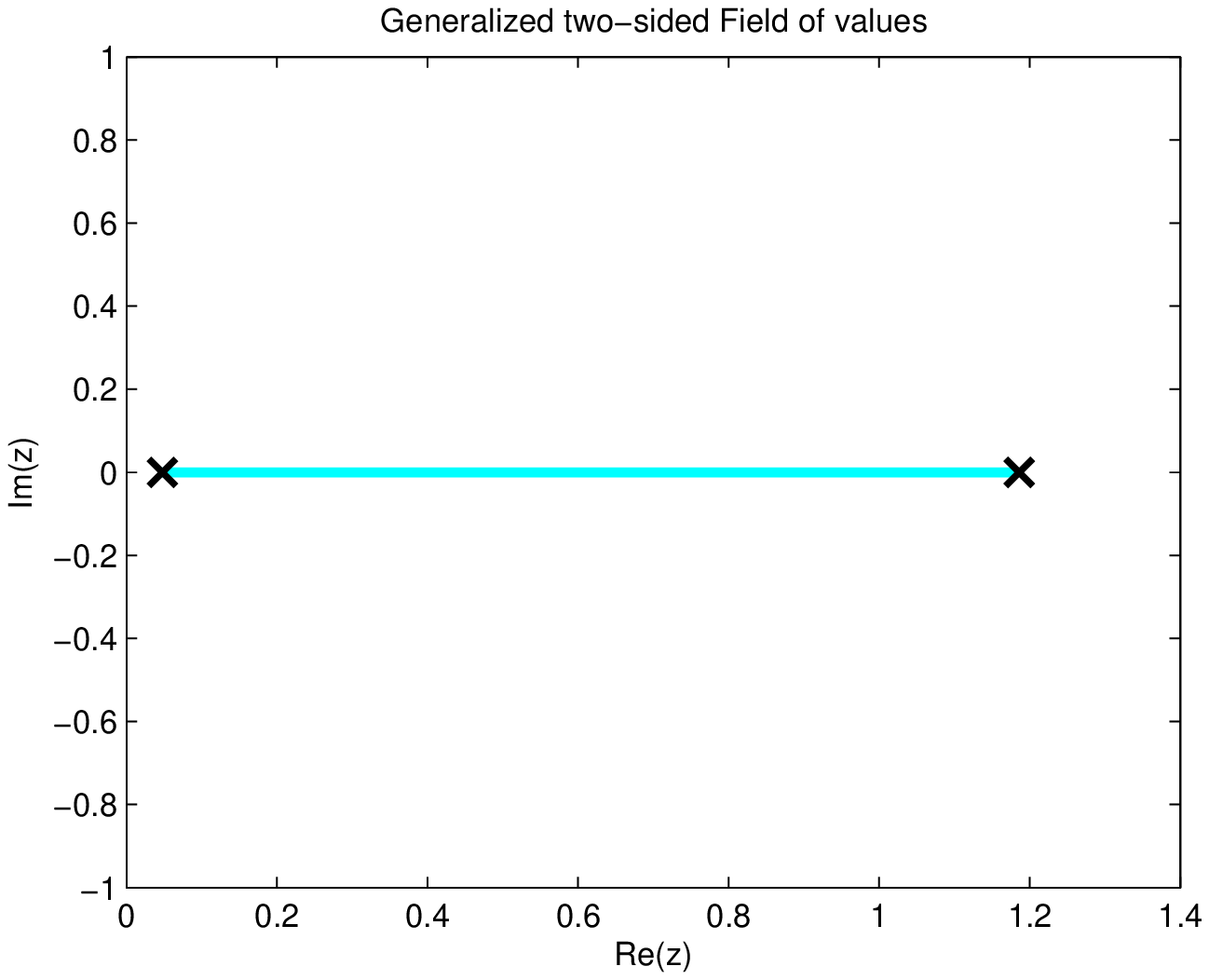}
		\caption{Gen. two-sided FoV, $G(A)$}
		\label{fig:fovalsrandA2}
	\end{minipage}
\end{figure}

\begin{figure}
	\begin{minipage}{0.5\linewidth}
		\centering
		\includegraphics[scale=0.4]{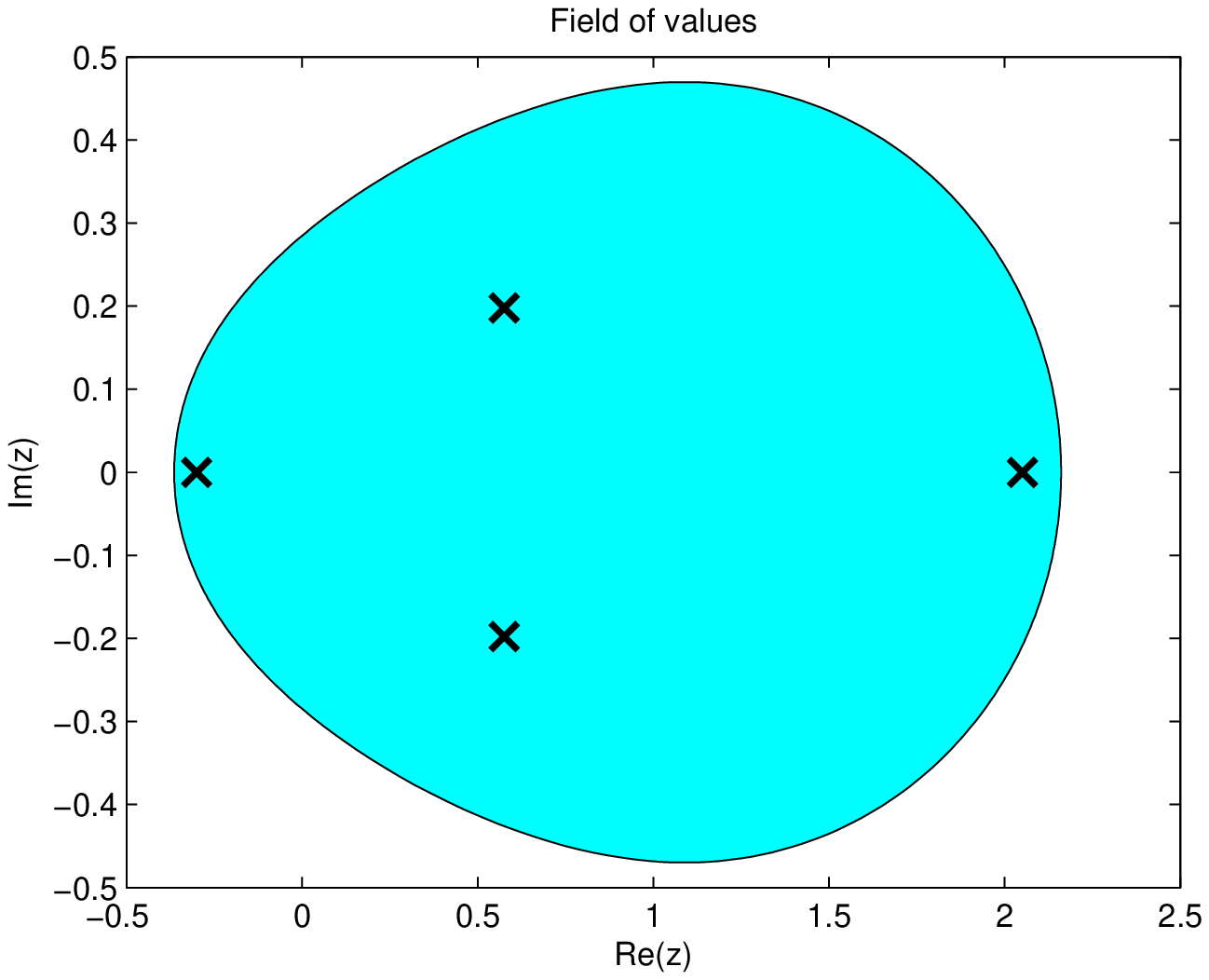}
		\caption{Standard FoV, $F(B)$}
		\label{fig:fovalsA3}
	\end{minipage}
	\begin{minipage}{0.5\linewidth}
		\centering
		\includegraphics[scale=0.4]{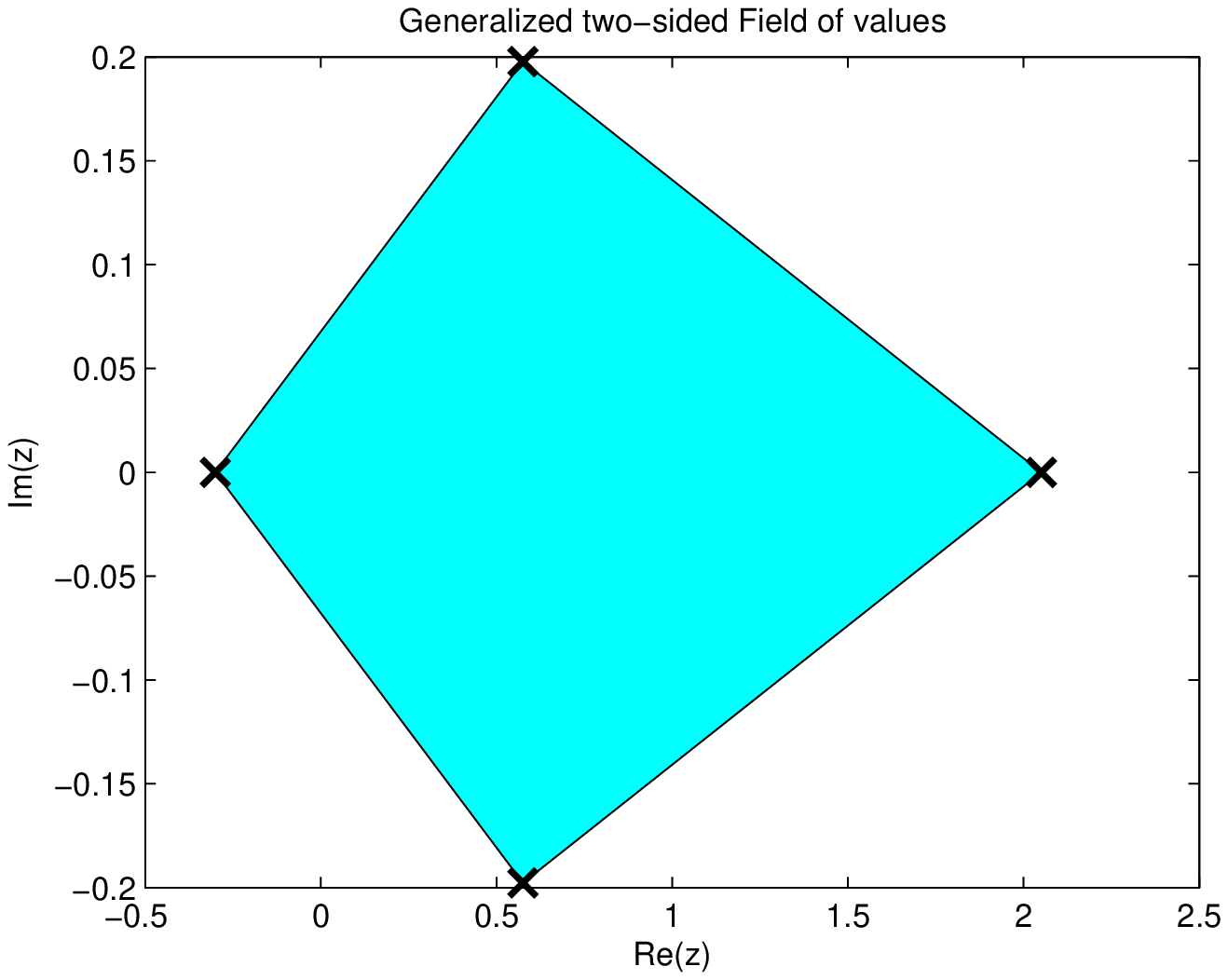}
		\caption{Gen. two-sided FoV, $G(B)$}
		\label{fig:fovalsrandA3}
	\end{minipage}
\end{figure}	
\end{example}

\section{Generalized two-sided Rayleigh Quotient}\label{sec:TSRQ}

Such as the standard field of values is related to the standard Rayleigh Quotient so the generalized field of values is connected with a generalized Rayleigh Quotient. Recall that given a Hermitian matrix $A$, and an $n$-dimensional vector $x$ the standard Rayleigh Quotient, $\rho(x,A)$ is the function
\begin{equation}\nonumber
	\rho(x,A)=\frac{x^{*}Ax}{x^{*}x},\hspace{10mm}x\neq 0.
\end{equation}
For the nonsymmetric case, the situation is more delicate. Instead of the quadratic form, the loss of symmetry requires us to handle a bilinear one. An intuitive generalization would be
\begin{equation}\label{eq:oldTSRQ}
	\tilde{\rho}(y,x,A)=\frac{y^{*}Ax}{y^{*}x},\hspace{10mm}y^{*}x\neq 0
\end{equation}
which is consistent with the ones given in \cite[In particular part III]{ostrowski_convergence_1957} and follow-ups such as \cite{ parlett_rayleigh_1974,hochstenbach_two-sided_2003}. A consequence of the loss of symmetry, however, is that with the generalization just defined and even if all the eigenvalues of $A$ would be real, $\tilde{\rho}(y,x,A)$ is no longer maximized at the largest eigenpair of $A$ (or, for what is worth, minimized at the smallest). The correct generalization requires an extra constraint.
\begin{definition}[Generalized two-sided Rayleigh Quotient]
	Assume the matrices $A,M\in\mathbb{C}^{n\times n}$ are nondefective and let $A$ be diagonalizable as $A=V\Lambda V^{-1}$. Denote by $y$ and $x$ two $n$-dimensional vectors in $\mathbb{C}$ such that $y^{*}x\neq 0$. The generalized two-sided Rayleigh Quotient of $x$ and $y$ is then defined as the function
	\begin{equation}\nonumber
		\rho(y,x,A)=\frac{y^{*}Ax}{y^{*}Mx},\hspace{10mm}\mbox{with}\hspace{3mm}y^{*}Mx\neq 0\hspace{3mm}\mbox{and}\hspace{3mm}VV^{*}y=x.
	\end{equation}
\end{definition}
The matrix $M$ in the denominator is a Hermitian and positive definite matrix in the $H$-inner product, where $H=(VV^{*})^{-1}$. In addition, the standard Rayleigh Quotient is obtained for particular choices of $x,y$ and $M$, namely with $VV^{*}=M=I$. For what follows, we assume, that $M=I$. 

\subsection{Properties of the Generalized two-sided Rayleigh Quotient}

It is a known fact that the standard Rayleigh Quotient, $\rho(v_{r},A)$ of a normalized right eigenvector, $v_{r}$, associated with eigenvalue $\lambda_{i}$ of a matrix $A$ satisfies $\rho(v_{r},A)=\lambda_{i}$. An equivalent statement is true for the left eigenvector $v_{l}$ associated with $\lambda_{i}$, that is, $\rho(v_{l},A)=\lambda_{i}$. Equivalently, for the generalized two-sided Rayleigh Quotient, $\rho(v_{l},v_{r},A)$
\begin{lemma}\label{lem:twosidedRQ}
Let $v_{r},v_{l}\in\mathbb{C}^{n}$ be the normalized right and left eigenvectors associated with the same eigenvalue $\lambda_{i}$ of a matrix $A\in\mathbb{C}^{n\times n}$. Define $\rho(v_{l},v_{r},A)$ as the generalized two-sided Rayleigh Quotient of $v_{l}$ and $v_{r}$. Then,
	\begin{equation}\nonumber
		\rho(v_{l},v_{r},A)=\frac{v_{l}^{*}Av_{r}}{v_{l}^{*}v_{r}}=\frac{v_{l}^{*}Av_{l}}{v_{l}^{*}v_{l}}=\frac{v_{r}^{*}Av_{r}}{v_{r}^{*}v_{r}}=\lambda_{i}.
	\end{equation}
\end{lemma}

{\em Proof}.
Notice that because $v_{r}$ and $v_{l}$ are the right and left eigenvectors corresponding to $\lambda_{i}$ the constraint $v_{l}=(VV^{*})^{-1}v_{r}$ is superfluous. Moreover, $v_{l}^{*}v_{r}\neq 0$. Assume $\|v_{l}\|=\|v_{r}\|=1$, then the last two equalities follow from the definition of right and left eigenvector. For if $v_{l}^{*}A=\lambda_{i}v_{l}^{*}$ and $Av_{r}=\lambda_{i}v_{r}$ then
\begin{equation}\nonumber
	v_{l}^{*}Av_{l}=\lambda_{i}\|v_{l}\|_{2}^{2}\hspace{5mm}\mbox{and}\hspace{5mm}v_{r}^{*}Av_{r}=\lambda_{i}\|v_{r}\|_{2}^{2}.
\end{equation}
As for the second equality, 
\begin{equation}\nonumber
	\rho(v_{l},v_{r},A)=\frac{v_{l}^{*}Av_{r}}{v_{l}^{*}v_{r}}=\frac{\lambda_{i}v_{l}^{*}v_{r}}{v_{l}^{*}v_{r}}=\lambda_{i}. \qquad\endproof
\end{equation}

Given a nonzero vector $x$, the standard Rayleigh Quotient is the scalar $\rho(x,x,A)$ for which $x\perp(A-\rho(x,x,A) I)x$. Likewise, the generalized two-sided Rayleigh Quotient is the scalar $\rho(y,x,A)$ for which
\begin{equation}\nonumber
	y\perp(A-\rho(y,x,A)I)x\hspace{10mm}\mbox{and}\hspace{10mm}x\perp (A^{*}-\overline{\rho(y,x,A)}I)y.
\end{equation}
The vector $\rho(y,x,A)x$ is the oblique projection onto $x$ and orthogonally to $y$ of $Ax$ while $\overline{\rho(y,x,A)}y$ is the oblique projection onto $y$ and orthogonally to $x$ of $A^{*}y$. Moreover, recall that given a vector $x\in\mathbb{C}^{n}$ with norm one and a scalar $\mu\in\mathbb{C}$, the measure of how close $(\mu,x)$ is of being an eigenpair of $A$ is $\|r\|_{2}^{2}=\langle r,r\rangle$ where $r=Ax-\mu x$. In truth, we need only $x$, as the scalar $\mu$ minimizing $\|r\|_{2}$ is no other than $\rho(x,x,A)$ (see \cite{parlett_rayleigh_1974} or \cite{parlett_symmetric_1998}). The non-Hermitian case is slightly more involved since the left and right eigenvectors differ. Therefore, a triplet $(\mu,x,y)$ consisting of a scalar and two $n$-vectors is needed to determine $r_{y}=y^{*}A-\mu y^{*}$ and $r_{x}=Ax-\mu x$. 

Standard properties of the classic Rayleigh Quotient such as homogeneity and translation invariance are carried over, in a straightforward way, to the generalized two-sided Rayleigh Quotient:
\begin{itemize}
	\item Homogeneity: $\rho(\alpha v, \beta w, A)=\rho(v,w,A)$ and
		\begin{equation}\nonumber
			\rho(v,w,\beta A)=\frac{v^{*}\beta Aw}{v^{*}Bw}=\beta\frac{v^{*}Aw}{v^{*}Bw}=\beta\rho(v,w,A)\mbox{;}
		\end{equation}
	\item Translation Invariance: $\rho(v,w,A-\mu I)=\rho(v,w,A)-\mu$.
\end{itemize}
In addition, the \emph{boundedness} property that fails when taking $\tilde{\rho}(y,x,A)$ (see \cite{parlett_rayleigh_1974}) is satisfied for the generalized two-sided version we propose (cf. Property \ref{prop:Comp} and \ref{prop:ConvexHull}). As is an equivalent to the minimal residue property of the standard Rayleigh Quotient. This also in contrast to the approach given by Equation (\ref{eq:oldTSRQ}) without the extra constraint. In this situation, we can only guarantee that for a \emph{nonsymmetric matrix $A$ with real eigenvalues}, the quantity $\tilde{\rho}(y,x,A)$ minimizes the inner product $\langle r_{y},r_{x}\rangle$.

\begin{itemize}
\item{Minimal Inner product :} 

For $A\in\mathbb{R}^{n\times n}$ with real eigenvalues and given $x,y\in\mathbb{R}^{n}$ such that $y^{*}x\neq 0$, for any scalar $\rho$
	\begin{equation}\nonumber
		(A^{*}y-\mu y^{*})^{*}(Ax-\mu x)\geq y^{*}A^{2}x-\rho^{2}y^{*}x,
	\end{equation}
 with equality only when $\mu=\rho=\rho(y,x,A)$.
\end{itemize}
\begin{proof}
	Set $\tilde{\rho}=\tilde{\rho}(y,x,A)$.
	\begin{eqnarray}\nonumber
		(A^{*}y-\mu y^{*})^{*}(Ax-\mu x)&=&		y^{*}AAx-\mu y^{*}Ax-\mu y^{*}Ax-\mu^{2}y^{*}x \nonumber \\
									&=&		y^{*}x\left( y^{*}A^{2}x/y^{*}x +(\mu-\rho)^{2}-\rho^{2} \right)\nonumber \\
									&\geq & y^{*}A^{2}x -\rho^{2}y^{*}x, \nonumber
	\end{eqnarray}
	with equality only when $\mu=\rho$.
\end{proof}

If, however, the extra constraint ($y=(VV^{*})^{-1}x$) is used 
\begin{lemma}[Minimal residue norm]
	Let $A=V\Lambda V^{-1}$ and $H=(VV^{*})^{-1}$. Given $u,v\neq 0$ and for any scalar $\rho$
	\begin{equation}\nonumber
		\|Au-\mu u\|_{H}^{2}\geq \|Au\|_{H}^{2}-\|\rho u\|_{H}^{2}\hspace{5mm}\mbox{and}\hspace{5mm}\|v^{*}A-\mu v^{*}\|_{H}^{2}\geq \|v^{*}A\|_{H}^{2}-\|\rho v^{*}\|_{H}^{2}
	\end{equation}
	with equality only when $\mu=\rho=\rho(y,x,A)$.
\end{lemma}

\begin{proof}
	Set $\rho=\rho(y,x,A)$, $H=(VV^{*})^{-1}$ and $r_{u}=Au-\mu u$ and recall that $H$ is a Hermitian positive definite matrix.
	\begin{eqnarray}
		\|r_{u}\|^{2}_{H}&=&(HAu-\mu Hu)^{*}(Au-\mu u)\nonumber \\
										&=& u^{*}A^{*}HAu-\bar{\mu}u^{*}HAu-\mu u^{*}A^{*}Hu+\bar{\mu}\mu u^{*}Hu \nonumber \\
										&=& (u^{*}Hu)\left[ (u^{*}A^{*}HAu)/(u^{*}Hu) +(\mu - \rho)(\bar{\mu}- \bar{\rho})-\rho\bar{\rho} \right] \nonumber \\
										&\geq & \|Au\|^{2}_{H}-\|\rho u\|^{2}_{H} \nonumber
	\end{eqnarray}
where $\|z\|^{2}_{H}=\langle z,z \rangle_{H}=\langle z,Hz\rangle$ for any $z\in\mathbb{C}^{n}$.  We replace $r_{u}$ in the proof by $r_{v}=v^{*}A-\mu v^{*}$ to obtain the second part of the statement.
\end{proof}

Because the range of $\rho(y,x,A)$ is a subset of the range of $\tilde{\rho}(y,x,A)$ we cite a result from B.N. Parlett \cite{parlett_rayleigh_1974} for the \emph{stationarity} property in the form of a lemma.
\begin{lemma}
The generalized two-sided Rayleigh Quotient, $\rho(y,x,A)$, is stationary if and only if $y$ and $x$ are the left and right eigenvectors of $A$ associated with eigenvalue $\rho(y,x,A)$ and $y^{*}x\neq 0$.
\end{lemma}
\begin{proof}
(see \cite[\S 11]{parlett_rayleigh_1974}).
\end{proof}

\subsection{Extrema of the generalized two sided Rayleigh Quotient}

Similar to the normal case, the extrema of the generalized two-sided Rayleigh quotient are the largest and the smallest eigenvalues of those non-Hermitian matrices whose eigenvalues are real. This is the topic of the next theorem whose Hermitian version is attributed to Rayleigh and Ritz.
\begin{theorem}\label{thm:RayRitz}
	Let $A\in\mathbb{C}^{n\times n}$ be a non-Hermitian, nondefective matrix. Let $\Lambda=V^{-1}AV$ be the diagonal matrix of eigenvalues of $A$ and assumed to be real and ordered as $\lambda_{1}\leq\ldots\leq\lambda_{n}$. Then,
	\begin{align}
		&\lambda_{1}y^{*}x\leq y^{*}Ax\leq \lambda_{n}y^{*}x\hspace{5mm}\mbox{for all}\hspace{5mm}x\in\mathbb{C}^{n}\hspace{5mm}\mbox{and}\hspace{5mm}y=(VV^{*})^{-1}x\\ 
		&\lambda_{max}=\lambda_{n}=\mathop{\max_{y^{*}x\neq 0}}_{y=(VV^{*})^{-1}x} \frac{y^{*}Ax}{y^{*}x}=\mathop{\max_{y^{*}x=1}}_{y^{*}=(VV^{*})^{-1}x}y^{*}Ax\\
		&\lambda_{min}=\lambda_{1}=\mathop{\min_{y^{*}x\neq 0}}_{y=(VV^{*})^{-1}x} \frac{y^{*}Ax}{y^{*}x}=\mathop{\min_{y^{*}x=1}}_{y^{*}=(VV^{*})^{-1}x}y^{*}Ax.
	\end{align}
	
\end{theorem}

{\em Proof}.
	From Lemma \ref{lem:rleigv} we know that if $v_{l}$ and $v_{r}$ are the left and right eigenvector of $A$ corresponding to the eigenvalue $\lambda_{i}$, then $(v_{l}^{*}V)^{*}=V^{-1}v_{r}=e_{i}$ or equivalently, $v_{l}=(VV^{*})^{-1}v_{r}$. Now, for any $x\in\mathbb{C}^{n}$ and $y=(VV^{*})^{-1}x$ 
	\begin{equation}\nonumber
		y^{*}Ax=y^{*}V\Lambda V^{-1}x=x^{*}(V^{-1})^{*}\Lambda V^{-1}x=\sum_{i=1}^{n}\lambda_{i}|(V^{-1}x)_{i}|^{2}.
	\end{equation}
	Because each term $|(V^{-1}x)_{i}|^{2}$ is nonnegative, this is a convex combination of the real numbers $\lambda_{i}$ for $i=1,\ldots,n$. Therefore
	\begin{equation}\nonumber
		\lambda_{min}\sum_{i=1}^{n}|(V^{-1}x)_{i}|^{2}\leq y^{*}Ax=\sum_{i=1}^{n}\lambda_{i}|(V^{-1}x)_{i}|^{2}\leq\lambda_{max}\sum_{i=1}^{n}|(V^{-1}x)_{i}|^{2}.
	\end{equation}
	For $z=V^{-1}x$ notice that $\displaystyle\sum_{i=1}^{n}|(V^{-1}x)_{i}|^{2}=\|z\|^{2}=z^{*}z=y^{*}x$. Consequently,
	\begin{equation}\nonumber
		\lambda_{1}y^{*}x\leq y^{*}Ax\leq \lambda_{n}y^{*}x.
	\end{equation}
Equality occurs when $y$ and $x$ are the right and left eigenvalues of $A$ corresponding to $\lambda_{1}$ or $\lambda_{n}$ as appropriate (see Lemma \ref{lem:twosidedRQ}). In other words,
	\begin{equation}\nonumber
		\mathop{\min_{y^{*}x\neq 0}}_{y=(VV^{*})^{-1}x}\frac{y^{*}Ax}{y^{*}x}=\lambda_{1}\hspace{10mm}\mbox{and}\hspace{10mm}
		\mathop{\max_{y^{*}x\neq 0}}_{y=(VV^{*})^{-1}x}\frac{y^{*}Ax}{y^{*}x}=\lambda_{n}.
	\end{equation}
	In addition $y$ and $x$ can be normalized so that $y^{*}x=1$ resulting in
	\begin{equation}\nonumber
		\mathop{\max_{y^{*}x=1}}_{y=(VV^{*})^{-1}x}y^{*}Ax=\lambda_{n}\hspace{10mm}\mbox{and}\hspace{10mm}\mathop{\min_{y^{*}x=1}}_{y=(VV^{*})^{-1}x}y^{*}Ax=\lambda_{1}.\qquad\endproof
	\end{equation}

Two additional notes are now called for. One to say that for $z=V^{-1}x$ and $\Lambda=V^{-1}AV$
\begin{equation}\nonumber
	\mathop{\max_{y^{*}x=1}}_{y=(VV^{*})^{-1}x}y^{*}Ax=\max_{\|z\|=1}z^{*}\Lambda z
\end{equation}
Therefore, by symmetry of $\Lambda $, what was just developed for $x$ can in a similar manner be done for $y$ by setting $x=VV^{*}y$. The second to draw attention to the fact that
\begin{equation}\nonumber
\mathop{\max_{y^{*}x\neq 0}}_{y=(VV^{*})^{-1}x}\frac{y^{*}Ax}{y^{*}x}\hspace{2mm}\leq\hspace{2mm}\max_{y^{*}x\neq 0}\frac{y^{*}Ax}{y^{*}x}\hspace{8mm}\mbox{and}\hspace{8mm}\mathop{\min_{y^{*}x\neq 0}}_{y=(VV^{*})^{-1}x}\frac{y^{*}Ax}{y^{*}x}\hspace{2mm}\geq\hspace{2mm}\min_{y^{*}x\neq 0}\frac{y^{*}Ax}{y^{*}x}
\end{equation}
as a result of the extra constraint on the left-hand side.

The extra restriction to the set over which the maximum is taken, can be seen as the nonnormal equivalent of the condition $y=x$, since in the normal case, $V$ is orthogonal and $VV^{*}=I$. Unfortunately the price to pay for nonnormality is high, rendering limited practical use to the previous results. The matrix $V$ is, in general, not known and if otherwise there would no longer be the need for determining right and left eigenvectors. In theoretical terms, however, it allows for the generalization of the variational characterization of the eigenvalues to non-Hermitian matrices with real eigenvalues. We are now able to generalize Courant-Fischer minimax theorem to real non-Hermitian matrices with real eigenvalues.

\begin{theorem}\label{thm:CFnonHermitian}
	Let $j$ and $n$ be integers such that $1\leq j\leq n$ and let $A\in\mathbb{C}^{n\times n}$ be a nondefective matrix. Assume the eigenvalues of $A$ to be real and ordered as $\lambda_{1}\leq\ldots\leq\lambda_{n}$. Let $\textbf{S}^{j}$ denote a $j$-dimensional subspace of $\mathbb{C}^{n}$. Then, 
	\begin{equation}\label{eq:minmax}
		\lambda_{j}=\min_{\textbf{S}^{j}}\mathop{\mathop{\max_{x\in\textbf{S}^{j}}}_{y=(VV^{*})^{-1}x}}_{y^{*}x\neq 0}\frac{y^{*}Ax}{y^{*}x}=\max_{\textbf{S}^{n-j+1}}\mathop{\mathop{\min_{q\in \textbf{S}^{n-j+1}}}_{p=(VV^{*})^{-1}q}}_{p^{*}q\neq 0}\frac{p^{*}Aq}{p^{*}q}
	\end{equation}
 \end{theorem}

{\em Proof}.
		We have shown earlier that with the variable transformation $z=V^{-1}x$
	\begin{equation}\nonumber
		\rho(y,x,A)=\mathop{\max_{y^{*}x\neq 0}}_{y=(VV^{*})^{-1}x} \frac{y^{*}Ax}{y^{*}x}=\max_{z\neq 0}\frac{z^{*}\Lambda z}{z^{*}z}=\rho(z,z,\Lambda).
 	\end{equation}
By the basis theorem and because the columns of $V$ form a basis for $\mathbb{C}^{n}$, though not orthogonal, $\tilde{\textbf{S}}^{j}\cap \tilde{\textbf{S}}^{n-j+1}\neq 0$. There exist, thus, a vector $z\in \tilde{\textbf{S}}^{j}\cap \tilde{\textbf{S}}^{n-j+1}$ for which
\begin{equation}\nonumber
	\min_{v\in\tilde{\textbf{S}}^{n-j+1}}\rho(v,v,\lambda)\leq\rho(w,w,\Lambda)\leq\max_{u\in\tilde{\textbf{S}}^{j}}\rho(u,u,\Lambda).
\end{equation}
Because the inequalities are valid for all choices of $\tilde{\textbf{S}}^{j}$ and $\tilde{\textbf{S}}^{n-j+1}$ we have
	\begin{equation}\nonumber
		\max_{\tilde{\textbf{S}}^{n-j+1}}\min_{v\in\tilde{\textbf{S}}^{n-j+1}}\rho(v,v,\Lambda)\leq\min_{\tilde{\textbf{S}}^{j}}\max_{u\in\tilde{\textbf{S}}^{j}}\rho(u,u,\Lambda).
	\end{equation}
In the basis of the original matrix this means
	\begin{equation}\nonumber
		\max_{\textbf{S}^{n-j+1}}\mathop{\min_{q\in \textbf{S}^{n-j+1}}}_{p=(VV^{*})^{-1}q}\rho(p,q,A)\leq\min_{\textbf{S}^{j}}\mathop{\max_{x\in \textbf{S}^{j}}}_{y=(VV^{*})^{-1}x}\rho(y,x,A).
	\end{equation}
Equality follows from taking $\textbf{S}^{j}=\mathcal{V}^{j}$ the span of the first $j$ (right) eigenvectors and $\textbf{S}^{n-j+1}$ the span of the last $n-j+1$ (right) eigenvectors, rendering 
	\begin{equation}\nonumber
		\lambda_{j}=\max_{\textbf{S}^{n-j+1}}\mathop{\min_{q\in \textbf{S}^{n-j+1}}}_{p=(VV^{*})^{-1}q}\rho(p,q,A)\hspace{10mm}\mbox{and}\hspace{10mm}\lambda_{j}=\min_{\textbf{S}^{j}}\mathop{\max_{x\in \textbf{S}^{j}}}_{y=(VV^{*})^{-1}x}\rho(y,x,A).\qquad\endproof
	\end{equation}

\section*{Acknowledgments}
The author wishes to thank Jan Brandts for his valuable comments on an early version of this paper.

\bibliography{FoVRRSilva}{}
\bibliographystyle{plain}
\end{document}